\documentclass[a4paper,11pt]{article}

\usepackage[a4paper,margin=1in]{geometry}

\usepackage{amsmath,amssymb,amsthm}
\usepackage{mathtools}
\usepackage{bm}
\usepackage{cite}
\usepackage{hyperref}
\usepackage{mathrsfs}
\usepackage{setspace}
\usepackage{float}
\usepackage{tikz}
\usepackage{subcaption}
\usetikzlibrary{shapes.geometric, arrows.meta, fit}
\usepackage{paralist}

\hypersetup{
colorlinks=true,
linkcolor=blue,
citecolor=blue,
urlcolor=blue
}

\theoremstyle{plain}
\newtheorem{theorem}{Theorem}[section]
\newtheorem{lemma}[theorem]{Lemma}
\newtheorem{proposition}[theorem]{Proposition}
\newtheorem{corollary}[theorem]{Corollary}

\theoremstyle{definition}
\newtheorem{definition}[theorem]{Definition}
\newtheorem{problem}[theorem]{Problem}
\newtheorem{example}[theorem]{Example}

\theoremstyle{remark}
\newtheorem{remark}[theorem]{Remark}
\newtheorem*{remark*}{Remark}

\numberwithin{equation}{section}
\allowdisplaybreaks[4]

\def\T{\mathrm T}
\def\rd{\mathrm d}
\def\e{\mathrm e}
\def\diag{\mathrm{diag}}
\def\BA{\mathbb A}
\def\BN{\mathbb N}
\def\BR{\mathbb R}
\def\BC{\mathbb C}
\def\BQ{\mathbb Q}
\def\BS{\mathbb S}
\def\cA{\mathcal A}
\def\cS{\mathcal S}
\def\cL{\mathcal L}
\def\Ga{\Gamma}
\def\Om{\Omega}
\def\al{\alpha}
\def\be{\beta}
\def\ga{\gamma}
\def\la{\lambda}
\def\vp{\varphi}
\def\f{\frac}
\def\pa{\partial}
\def\wt{\widetilde}

\tikzset{
base/.style={circle, draw=black, thick, minimum size=0.8cm, font=\large\bfseries, inner sep=0pt},
arrow/.style={-{Stealth[scale=1.2]}, thick, draw=black!90},
active_node/.style={base, draw=black, fill=gray!30, line width=1.2pt},
box_style/.style={rectangle, draw=black, thick, dashed, inner sep=5pt},
circle_box/.style={circle, draw=red, line width=1.5pt, inner sep=1.5pt},
square_box/.style={rectangle, draw=red, line width=1.5pt, inner sep=5pt},
triangle_box/.style={regular polygon, regular polygon sides=3, draw=red, line width=1.5pt, inner sep=-1pt},
hexagon_box/.style={regular polygon, regular polygon sides=6, draw=red, line width=1.5pt, inner sep=2pt}
}

\title{\bf Strict Positivity Property of Inhomogeneous\\
Subdiffusion Equations and Its Application\\
to Coupled Subdiffusion Systems}

\author{Yimeng Tian\thanks{School of Mathematics, Shanghai University of Finance and Economics, Shanghai 200433, China. E-mail: tian.yimeng@stu.sufe.edu.cn}
\and 
Yikan Liu\thanks{Department of Mathematics, Kyoto University, Kitashirakawa-Oiwakecho, Sakyo-ku, Kyoto 606-8502, Japan. E-mail: liu.yikan.8z@kyoto-u.ac.jp}}

\date{}

\begin{document}

\maketitle

\begin{abstract}
The positivity of solutions to subdiffusion equations has been widely studied, mainly in the context of homogeneous problems for single equations. In this article, we fill the missing strict positivity for inhomogeneous subdiffusion equations with nonnegative and nontrivial source terms by connecting Green's functions for fractional and classical diffusion equations via special functions. As a direct application, we further investigate the strict positivity property of coupled subdiffusion systems with nonnegative and partially nontrivial initial values or sources. Under suitable cooperativeness and connectivity conditions, the strict positivity turns out to propagate not only in time but also across different components of the system, reflecting the intrinsic interactions induced by the coupling structure. These results provide a unified framework for understanding positivity properties of both scalar and coupled subdiffusion equations, offering new insights beyond the classical maximum principle approach.
\end{abstract}



\section{Introduction}\label{sec-intro}

Let $T>0,\al\in(0,1)$ be constants and $\Om\subset\BR^d$ ($d\in\BN:=\{1,2,\dots\}$) be a bounded domain with a smooth boundary $\pa\Om$. We define the fractional derivative of order $\al$ in time, denoted by $\pa_t^\al$, as the inverse operator of the $\al$-th order Riemann-Liouville integral operator
\[
J^\al:L^2(0,T)\longrightarrow L^2(0,T),\quad J^\al f(t):=\int_0^t\f{s^{\al-1}}{\Ga(\al)}f(t-s)\,\rd s,
\]
where $\Ga(\,\cdot\,)$ is the Gamma function. Consider the initial-boundary value problem for a single subdiffusion equation
\begin{equation}\label{Eq:IBVP}
\begin{cases}
\pa_t^\al(u-a)+\cA u=F & \mbox{in }\Om\times(0,T),\\
u=0 & \mbox{on }\pa\Om\times(0,T).
\end{cases}
\end{equation}
Here $\cA:D(\cA)\longrightarrow L^2(\Om)$ denotes the elliptic operator with domain $D(\cA):=H^2(\Om)\cap H_0^1(\Om)$ defined by
\begin{equation}\label{Eq:A-opt}
\cA f:=-\mathrm{div}(\bm A(\bm x)\nabla f)+c(\bm x)f=-\sum_{i,j=1}^d\pa_j(a_{ij}(\bm x)\pa_i f)+c(\bm x)f,\quad f\in D(\cA),
\end{equation}
where $\bm A\in C^1(\overline\Om;\BR^{d\times d}_{\mathrm{sym}})$ is a symmetric matrix-valued function and $0\le c\in C(\overline\Om)$. Moreover, we assume the uniform ellipticity condition, i.e., there exists a constant $\theta>0$ such that
\[
\theta|\bm\xi|^2\le\bm\xi\cdot\bm A(\bm x)\bm\xi,\quad\forall\,\bm x\in\overline\Om,\ \forall\,\bm\xi\in\BR^d,
\]
where $\cdot$ and $|\cdot|$ stand for the Euclidean inner product and norm on $\BR^d$, respectively. Moreover, $a$ is an $\bm x$-dependent function taking the role of the initial value in the sense that $u-a$ lies in the domain of $\pa_t^\al$ for a.e.\! $\bm x\in\Om$, while $F=F(\bm x,t)$ is the source term. The precise assumptions on $a$ and $F$ will be specified later.

Subdiffusion equations with fractional derivatives in time provide an effective mathematical framework for describing anomalous diffusion processes with memory effects arising in heterogeneous media and groundwater transport (see e.g. \cite{AG92,GCR92,MK00}). Within the last decades, fundamental analysis concerning equation \eqref{Eq:IBVP}, such as solution representations and well-posedness theory, has been mostly accomplished in abundant literature including \cite{EK04,L10,SY11,GLY15,J21,KRY20}, to list just a few. Concerning numerical analysis for subdiffusion equations, finite difference and finite element methods have also been extensively investigated (see, for example, \cite{LX07,JLZ13,JZ23}). Moreover, related inverse problems have also attracted considerable attention in recent years, see surveys \cite{LiuLiY19,LiLiuY19,LiY19} and the monograph \cite{KR23}.

Beyond the aforementioned aspects, the maximum principle is one of the most remarkable qualitative properties of parabolic-type equations, which plays a fundamental role in both the theoretical analysis of subdiffusion equations and the study of related inverse problems. For subdiffusion equations with Caputo derivatives, the weak maximum principle was first established in Luchko \cite{L09}. Then the strong maximum principle for the corresponding homogeneous equation was later derived in Zacher \cite{Z13} as a byproduct of a weak Harnack inequality and also independently in Liu et al. \cite{LRY16}. We also refer to the survey paper \cite{LuchkoY19} for a comprehensive overview of maximum principles for time-fractional partial differential equations and their various extensions. Recently, comparison principles for subdiffusion equations with general boundary conditions were founded in \cite{LY23,LY25} for both linear and semilinear cases. It is worth noting that most existing studies are concerned with maximum principles for homogeneous equations. Comparatively little attention has been paid to the strict positivity property of inhomogeneous subdiffusion equations and a unified analytical framework is still lacking. Nevertheless, the propagation of strict positivity for inhomogeneous problems constitutes an indispensable ingredient in establishing positivity results for coupled subdiffusion systems.

Motivated by this observation, we first investigate the strict positivity propagation for the inhomogeneous subdiffusion equation.

\begin{problem}\label{prob:single_source}
Let $u$ be the solution to the inhomogeneous problem
\begin{equation}\label{Eq:single(a=0)}
\begin{cases}
\pa_t^\al u+\cA u=F & \mbox{in } \Om\times(0,T),\\
u=0 & \mbox{on }\pa\Om\times(0,T).
\end{cases}
\end{equation}
If the source term $F$ is nonnegative and not identically zero, then can one conclude some strict positivity of the solution $u$?
\end{problem}

Furthermore, in many practical applications such as interacting population dynamics, ecological systems and biological growth processes, a single subdiffusion equation is often insufficient to describe the underlying mechanisms. This naturally motivates the study of the following initial-boundary value problem for a weakly coupled subdiffusion system of $K\in\BN$ components, which becomes the second object of this article:
\begin{equation}\label{Eq:Coupled_prototype}
\left\{\begin{alignedat}{2}
& \pa_t^{\al_k}(u_k-a_k)+\cA_k u_k+\sum_{\ell=1}^K c_{k\ell}u_\ell=F_k & \quad & \mbox{in }\Om\times(0,T),\\
& u_k=0 & \quad & \mbox{on }\pa\Om\times(0,T),
\end{alignedat}\right.\quad k=1,\dots,K.
\end{equation}
Here $\al_1,\dots,\al_K$ are constants satisfying $1>\al_1\ge\cdots\ge\al_K>0$, and each $\cA_k$ is an elliptic operator similarly to $\cA$ in \eqref{Eq:IBVP}. Further, $c_{k\ell}$ ($k,\ell=1,\dots,K$) are $\bm x$-dependent coupling coefficients, and $F_k=F_k(\bm x,t)$ is the source term of $u_k$. For convenience, we adopt a vector representation
\[
\bm u:=(u_1,\dots,u_K)^\T,\quad\bm a:=(a_1,\dots,a_K)^\T,\quad\bm F:=(F_1,\dots,F_K)^\T,
\]
and introduce formal matrices
\[
\pa_t^{\bm\al}:=\diag(\pa_t^{\al_1},\dots,\pa_t^{\al_K}),\quad\BA:=\diag(\cA _1,\dots,\cA_k),\quad\bm C:=(c_{k\ell})_{1\le k,\ell\le K}
\]
to rewrite \eqref{Eq:Coupled_prototype} concisely as
\begin{equation}\label{Eq:Coupled_prototype_matrix}
\begin{cases}
\pa_t^{\bm\al}(\bm u-\bm a)+(\BA+\bm C)\bm u=\bm F & \mbox{in }\Om\times(0,T),\\
\bm u=\bm0 & \mbox{on }\pa\Om\times(0,T).
\end{cases}
\end{equation}
Then $\BA u$ and $\bm C u$ are the principal and zeroth order terms in the spatial direction respectively, where $\bm C$ is the $\bm x$-dependent coefficient matrix coupling components in $\bm u$.

For \eqref{Eq:Coupled_prototype_matrix}, Li et al. \cite{LHL23} constructed a general framework for mild solutions and proved the corresponding well-posedness results. However, the strict positivity and positivity propagation remain far less understood than in the scalar setting. A non-negativity result for coupled subdiffusion systems was recently obtained in \cite{LY25}. More recently, BenSalah and Liu \cite{BL26} established a strong positivity property in the sense that, under suitable assumptions, a certain Riemann-Liouville integral of the solution is strictly positive. However, this result does not directly imply the strict positivity of the solution itself. Hence, we also investigate the following problem regarding the strict positivity property for the coupled system \eqref{Eq:Coupled_prototype_matrix}.

\begin{problem}\label{prob:coupled}
Let $\bm u$ satisfy \eqref{Eq:Coupled_prototype_matrix}, where all components in the initial value $\bm a$ and the source term $\bm F$ are nonnegative. If some components of $\bm a$ or $\bm F$ do not vanish identically, then can one conclude some strict positivity of $\bm u$ under certain assumptions?
\end{problem}

Building upon the strict positivity propagation established for scalar inhomogeneous subdiffusion equations, we investigates positivity propagation for both homogeneous and inhomogeneous coupled subdiffusion systems. In particular, we provide affirmative answers to Problem~\ref{prob:coupled}, thereby establishing a unified framework for strict positivity properties in both scalar equations and coupled systems.

The remainder of this paper is organized as follows.
In Section \ref{sec-preli}, we introduce several auxiliary lemmas that will be used in the proofs of the main results presented in Section \ref{sec-single} and \ref{sec-couple}. 
Section \ref{sec-single} is devoted to the study of strict positivity for a single equation, where we consider separately the homogeneous and inhomogeneous problems. Correspondingly, we establish Theorem~\ref{thm:SMP(F=0)} and Theorem~\ref{thm:SMP(a=0)} which address these scenarios.
As a direct application of the strict positivity result for the fractional inhomogeneous equation i.e. Theorem~\ref{thm:SMP(a=0)}, Section \ref{sec-couple} investigates the time propagation of positivity in the coupled system for both homogeneous and inhomogeneous problems. In particular, we establish Theorem~\ref{thm:coupled_F=0} and Theorem~\ref{thm:time-propagation}. The key ingredient in the analysis is Proposition~\ref{prop:coupled}, which builds a connection between the adjacency matrix of a graph and the index-set formulation of connectivity.
Finally, Section \ref{sec-concl} concludes the paper with a brief discussion.


\section{Preliminaries}\label{sec-preli}

We start with introducing some notation and recalling several special functions that will be used throughout the paper. The Mittag-Leffler function and the Wright function are defined as (see, e.g., \cite{GKMR14,M10})
\[
\begin{aligned}
E_{\al,\be}(z) & :=\sum_{k=0}^\infty\f{z^k}{\Ga(\al k+\be)},\\
W_{\al,\be}(z) & :=\sum_{k=0}^\infty\f{z^k}{k!\,\Ga(\al k+\be)},
\end{aligned}\quad z\in\BC,\ \al>0,\ \be\in\BC.
\]
It is known that both functions are entire on the complex plane. Next, we present a lemma that will be useful for our subsequent analysis.

\begin{lemma}\label{lem:M-nonnegtive}
For $0<\al<1$, the Wright function $M_\al(r)$ is nonnegative for $r>0$.
\end{lemma}

\begin{proof}
We recall the relation between the Mittag-Leffler function and the Wright function (see \cite[p.~249]{M10}):
\[
E_{\al,1}(-\eta)=\int_0^\infty\e^{-\eta r}M_\al(r)\,\rd r,\quad\eta\ge0.
\]
It is well known that $E_{\al,1}(-\eta)$ is completely monotone for $0<\al<1$. Hence, by Bernstein's theorem (see \cite[p.~3]{SSV12}), there exists a nonnegative measure $\mu$ such that
\[
E_{\al,1}(-\eta)=\int_0^\infty\e^{-\eta r}\,\rd\mu(r).
\]
Comparing the two representations and using the uniqueness of the Laplace transform, we readily see that $M_\al(r)$ is the Radon-Nikodym derivative of $\mu$ with respect to $r$ and hence $M_\al(r)\ge0$ for $r>0$.
\end{proof}

\begin{remark*}
The non-negativity of $M_\al$ can also be inferred from the fact that $\exp(-s^\al)$ is the Laplace transform of a probability density, namely the extremal L\'evy stable distribution. See \cite[p.~249]{M10} for further details.
\end{remark*}

We now specify function spaces and the spectral properties of the elliptic operators. Let $(\,\cdot\,,\,\cdot\,)$ be the inner product of $L^2(\Om)$ and $H_0^1(\Om)$, $H^2(\Om)$, etc.\! denote the standard Sobolev spaces \cite{A75}. Let $\{(\la_n,\vp_n)\}_{n=1}^\infty$ be the eigensystem of the self-adjoint elliptic operator $\cA$ defined in \eqref{Eq:A-opt} such that $0<\la_1<\la_2\le\cdots$ and $\la_n\longrightarrow\infty$ as $n\to\infty$. The corresponding eigenfunctions $\{\vp_n\}_{n=1}^\infty \subset H^2(\Om)\cap H_0^1(\Om)$ forms a complete orthonormal basis of $L^2(\Om)$. We define the fractional operator $\cA^\ga$ for $\ga\ge0$ and its associated domain and norm as
\begin{align*}
D(\cA^\ga) & :=\left\{f\in L^2(\Om)\mid\|f\|_{D(\cA^\ga)}:=\|\cA^\ga f\|_{L^2(\Om)}<\infty\right\},\\
\cA^\ga f & :=\sum_{n=1}^\infty\la_n^\ga(f,\vp_n)\vp_n,
\end{align*}
In a parallel manner, we extend the spectral framework to the coupled system. For each $k=1,\dots,K$, let $\{(\la_n^{(k)},\vp_n^{(k)})\}_{n=1}^\infty$ denote the eigensystem of $\cA_k$ similarly as before, from which we again define the fractional operator $\cA_k^\ga$ as well as its domain $D(\cA_k^\ga)$ for $\ga\ge0$. Consequently, we define the diagonal operator $\BA^\ga:=\diag(\cA_1^\ga,\dots,\cA_K^\ga)$ and $D(\BA^\ga)$ with its norm $\|\cdot\|_{D(\BA^\ga)}$.

We now recall the fundamental well-posedness results and the weak positivity property for problem \eqref{Eq:IBVP}.

\begin{lemma}\label{Lem:scalar_well-posedness}
Let $a\in D(\cA^\be)$ and $F\in L^p(0,T;D(\cA^\be))$ with $\be\ge0$ and $p\in[1,\infty]$.
\begin{enumerate}
\item If $F\equiv0,$ then there exists a unique solution $u$ to \eqref{Eq:IBVP} such that
\[
u(t)=\sum_{n=1}^\infty E_{\al,1}(-\la_n t^\al)(a,\vp_n)\vp_n\quad\mbox{in }\bigcap_{0\le\ga\le 1}L^{1/\ga}(0,T;D(\cA^{\be+\ga})),
\]
where we understand $1/\ga=\infty$ if $\ga=0$. Moreover, there exist a constant $C>0$ depending only on $\Om,\al,\cA$ such that
\[
\|u(t)\|_{D(\cA^{\be+\ga})}\le C\|a\|_{D(\cA^\be)}t^{-\al\ga},\quad0<t<T.
\]
\item If $a\equiv0,$ then there exists a unique solution $u$ to \eqref{Eq:IBVP} such that
\[
u(t)=\int_0^t(t-s)^{\al-1}\sum_{n=1}^\infty E_{\al,\al}(-\la_n(t-s)^\al)(F(s),\vp_n)\vp_nF(s)\,\rd s\quad\mbox{in }L^p(0,T;D(\cA^{\be+\ga}))
\]
for any $\ga\in[0,1)$. Moreover, there exists a constant $C>0$ depending only on $\Om,\al,\cA,T,\ga$ such that
\[
\|u\|_{L^p(0,T;D(\cA^{\be+\ga}))}\le C\|F\|_{L^p(0,T;D(\cA^\be))}.
\]
\item If $a\ge0$ in $\Om$ and $F\ge0$ in $\Om\times(0,T),$ then the solution $u\ge0$ in $\Om\times(0,T)$.
\end{enumerate}
\end{lemma}

The assertions in Lemma~\ref{Lem:scalar_well-posedness}(i)--(ii) summarize standard well-posedness results obtained in \cite{SY11,LLY15}. Meanwhile, Lemma~\ref{Lem:scalar_well-posedness}(iii) turns out to be a special case of the non-negativity result established in \cite{L09}.

Next, we turn to the coupled system \eqref{Eq:Coupled_prototype_matrix} and recall the definition of its mild solution.

\begin{definition}(see~\cite{LHL23})
Let $\bm C\in L^\infty(\Om)$, $\bm a\in L^2(\Om)$ and $\bm F\in L^p(0,T;L^2(\Om))$ with $p\in[1,\infty]$. We call $\bm u$ a mild solution to the initial-boundary value problem \eqref{Eq:Coupled_prototype_matrix} if it satisfies the integral equation
\[
\bm u=\bm w+\BQ\bm u\quad\mbox{in }\Om\times(0,T),
\]
where
\begin{gather*}
\bm w(t):=\BS(t)\bm a-\int_0^t \BA^{-1}\BS'(t-\tau)\bm F(\tau)\,\rd\tau,\\
\BQ\bm u(t):=\int_0^t\BA^{-1}\BS'(t-\tau)\bm C\bm u(\tau)\,\rd\tau.
\end{gather*}
Here we denote
\[
\BS(t):=\diag(\cS_1(t),\dots,\cS_K(t)),\quad\BA^{-1}\BS'(t):=\diag(\cA_1^{-1}\cS_1'(t),\dots,\cA_K^{-1}\cS_K'(t)),
\]
where the solution operator $\cS_k(t):L^2(\Om)\longrightarrow L^2(\Om)$ and its formal derivative $\cS'_k(t):L^2(\Om)\longrightarrow L^2(\Om)$ ($k=1,\dots,K$) are defined by 
\begin{align*}
\cS_k(t)\psi & :=\sum_{n=1}^\infty E_{\al_k,1}(-\la_n^{(k)}t^{\al_k})(\psi,\vp_n^{(k)})\vp_n^{(k)},\\
\cS_k'(t)\psi & :=-t^{\al_k-1}\sum_{n=1}^\infty E_{\al_k,\al_k}(-\la_n^{(k)}t^{\al_k})(\psi,\vp_n^{(k)})\vp_n^{(k)}.
\end{align*}
\end{definition}

For the (strict) positivity of vectors and matrices, we adopt the same notation as that in \cite[Definition 1]{BL26}, e.g., we write $\bm x\ge\bm0$ ($\bm x>\bm0$) for a vector $\bm x\in\BR^d$ if all components of $\bm x$ are nonnegative (strictly positive). We state a coupled counterpart of Lemma \ref{Lem:scalar_well-posedness}, i.e., the well-posedness and the weak maximum principle for \eqref{Eq:Coupled_prototype_matrix}.

\begin{lemma}\label{Lem:vector_well-posedness}
Let $\bm C\in L^\infty(\Om)$, $\bm a\in D(\cA^\be)$ and $\bm F\in L^p(0,T;D(\cA^\be))$ with $\be\ge0$ and $p\in[1,\infty]$.
\begin{enumerate}
\item If $\bm F\equiv\bm0,$ then there exists a unique mild solution $\bm u\in L^{1/\ga}(0,T;D(\BA^{\be+\ga}))$ to \eqref{Eq:Coupled_prototype_matrix} for any $\ga\in[0,1]$. Moreover, there exists a constant $C>0$ depending only on $\Om,\bm\al,\BA,\bm C$ such that
\[
\|\bm u(t)\|_{D(\BA^{\be+\ga})}\le C\|\bm a\|_{D(\BA^\be)}t^{-\al_1\ga},\quad0<t<T.
\]
\item If $\bm a\equiv\bm0,$ then there exists a unique mild solution $\bm u\in L^p(0,T;D(\BA^{\be+\ga}))$ to \eqref{Eq:Coupled_prototype_matrix} for any $\ga\in[0,1)$. Moreover, there exists a constant $C>0$ depending only on $\Om,\bm\al,\BA,\bm C,T,\ga$ such that
\[
\|\bm u\|_{L^p(0,T;D(\BA^{\be+\ga}))}\le C\|\bm F\|_{L^p(0,T;D(\BA^\be))}.
\]
\item Let $\bm a\ge\bm0$ in $\Om$ and $\bm F\ge\bm0$ in $\Om\times(0,T)$. If the coefficient matrix $\bm C$ satisfy the cooperativeness condition
\begin{equation}\label{eq-coop}
c_{k\ell}\le0\quad\mbox{ in }\Om,\forall\,k\ne\ell, 
\end{equation}
then $\bm u\ge\bm0$ in $\Om\times(0,T)$.
\end{enumerate}
\end{lemma}

Lemma~\ref{Lem:vector_well-posedness}(i)--(ii) follow from the well-posedness theory developed in \cite[Theorem 1]{LHL23}, while Lemma~\ref{Lem:vector_well-posedness}(iii) follows from the weak maximum principle established in \cite[Proposition 2]{BL26}.


\section{Strict Positivity of Single Equations}\label{sec-single}

In this section, we establish Theorem~\ref{thm:SMP(F=0)}, Theorem~\ref{thm:SMP(a=0)} and Corollary~\ref{coro:strict-positive} which address the homogeneous problem ($F=0$), the inhomogeneous problem ($a=0$) and the general case, respectively.

The key idea for establishing strict positivity is to show that the fractional equation inherits the positivity-preserving property of the classical heat equation. Rather than proving the strict positivity of a subdiffusion equation directly, we instead construct a connection between the Green's functions of a heat equation and a subdiffusion equation. Through this connection, the subdiffusion equation inherits the infinite propagation speed of a heat equation, which naturally yields its strict positivity. The crucial ingredient in this construction is the Wright function.

We begin with the inhomogeneous case driven by a source term, which has not been thoroughly investigated in the existing literature. The key ingredient is to establish the strict positivity of the Green's function associated with the inhomogeneous fractional equation, as shown in the following theorem.

\begin{theorem}\label{thm:1}
Let $0<\al\le1$ and let $G_\al(\bm x,\bm y,t)$ be the Green's function associated with the inhomogeneous problem defined by
\[
G_\al(\bm x,\bm y,t)=\sum_{n=1}^\infty t^{\al-1}E_{\al,\al}(-\la_n t^\al)\vp_n(\bm x)\vp_n(\bm y),\quad\mbox{a.e.\! }\bm x,\bm y\in\Om,\ t\in(0,T).
\]
Then $G_\al>0$ a.e.\! in $\Om\times\Om\times(0,T)$.
\end{theorem}

\begin{proof}
The proof is based on the strict positivity of Green's function for the classical heat equation with a homogeneous Dirichlet boundary condition. We then employ the Laplace transform to construct a bridge between classical and fractional Green's functions through a singular integral and the Wright function. This relation enables the fractional equation to inherit the positivity of the heat equation.

To begin with, we invoke the strict positivity of Green's function of the heat equation with a homogeneous Dirichlet boundary condition, that is,
\[
G_1(\bm x,\bm y,t):=\sum_{n=1}^\infty\e^{-\la_n t}\vp_n(\bm x)\vp_n(\bm y)>0,\quad\mbox{a.e.\! }\bm x,\bm y\in\Om,\ t\in (0,T).
\]
On the fractional side, we set $g^\al_n(t):=t^{\al-1}E_{\al,\al}(-\la_n t^\al)$ for simplicity to express Green's function for the subdiffusion equation as
\[
G_\al(\bm x,\bm y,t):=\sum_{n=1}^\infty g^\al_n(t)\vp_n(\bm x)\vp_n(\bm y).
\]

To reveal the relation between the heat equation and the fractional equation, we employ the Laplace transform of $g^\al_n$ as a bridge. Recall the following Laplace transform formula for the Mittag-Leffler function (see \cite[p.~12, Eq.~(1.42)]{M10}):
\[
\cL\left(t^{\be-1}E_{\al,\be}(-\la t^\al)\right)(s)=\f{s^{\al-\be}}{s^\al+\la},\quad s>0.
\]
For later use, we further analyze an improper integral to represent the Laplace-transformed result as
\[
\cL g^\al_n(s)=\f1{s^\al+\la_n}=\int_0^\infty\e^{-(\la_n+s^\al)\tau}\,\rd\tau,\quad s>0.
\]
We now apply the inverse Laplace transform to represent $g^\al_n$ via the Wright function as
\begin{align*}
g^\al_n(t) & =\cL^{-1}\left(\int_0^\infty\e^{-(\la_n+s^\al)\tau}\,\rd\tau\right)(t)=\int_0^\infty\e^{-\la_n\tau}\cL^{-1}\left(\e^{-s^\al\tau}\right)(t)\,\rd\tau\\
& =\int_0^\infty\e^{-\la_n\tau}\phi_\al(t,\tau)\,\rd\tau,
\end{align*}
where we defined $\phi_\al(t,\tau):=\cL^{-1}(\e^{-s^\al\tau})(t)$ and then utilized the Laplace inverse transform of $\e^{-r s^\al}$ (see \cite[p.~247, Eq.~(18)]{L19})
\[
\cL^{-1}\left(\e^{-r s^\al}\right)(t)=\f{\al r}{t^{\al+1}}M_\al(r t^{-\al}),\quad r>0.
\]
Furthermore, it follows from Lemma \ref{lem:M-nonnegtive} that 
\[
\phi_\al(t,\tau)=\al\tau t^{-\al-1}M_\al(\tau t^{-\al})\ge0
\]
and thus $g^\al_n(t)>0$. Therefore, substituting $g_n^\al(t)$ back into $G_\al(\bm x,\bm y,t)$ yields
\begin{align*}
G_\al(\bm x,\bm y,t) & =\sum_{n=1}^\infty\int_0^\infty\e^{-\la_n\tau}\phi_\al(t,\tau)\,\rd\tau\vp_n(\bm x)\vp_n(\bm y)\\
& =\int_0^\infty\sum_{n=1}^\infty\e^{-\la_n\tau}\vp_n(\bm x)\vp_n(\bm y)\phi_\al(t,\tau)\,\rd\tau\\
& =\int_0^\infty G_1(\bm x,\bm y,\tau)\phi_\al(t,\tau)\,\rd\tau>0.
\end{align*}
Here, since $G_1(\bm x,\bm y,\tau)>0$ and $\phi_\al(t,\tau)\ge0,\not\equiv0$, Tonelli's theorem (see \cite[Theorem 2.37, p.~67]{F99}) enables the interchange of the order of integration. This completes the proof.
\end{proof}

We now apply the strict positivity of the Green's function established above to prove the strict positivity of the solution. 
The following theorem provides the strict positivity of the solution to the single source problem, that is, the solution becomes strictly positive from the moment the source term appears.

\begin{theorem}\label{thm:SMP(a=0)}
Let $a=0$ and $F\in L^2(0,T;D(\cA^\be)$ with $\be>d/4-1$. Assume $F\ge0,\not\equiv0$ and define $t_0:=\inf\{t\ge0\mid F(\,\cdot\,,t)\not\equiv0\mbox{ in }\Om\}$. Then the solution $u$ to \eqref{Eq:IBVP} satisfies $u>0$ in $\Om\times(t_0,T)$.
\end{theorem}

\begin{proof}
By the definition of $G_\al$ and Lemma \ref{Lem:scalar_well-posedness}(ii), the solution to \eqref{Eq:IBVP} allows the representation
\[
u(\bm x,t)=\int_0^\T\!\!\!\int_\Om G_\al(\bm x,\bm y,t-s)F(\bm y,s)\,\rd\bm y\rd s.
\]
From Theorem \ref{thm:1}, we know that $G_\al>0$ a.e.\! in $\Om\times\Om\times(0,T)$.

Now fix $t>t_0$ arbitrarily. By the definition of $t_0$, we know $F\not\equiv0$ in $\Om\times(t_0,t)$. Since $F\ge0$, there exists a measurable set $E\subset\Om\times(t_0,t)$ with a positive measure such that $F>0$ in $E$. Owing to Theorem \ref{thm:1}, it reveals that
\[
G_\al(\bm x,\bm y,t-s)F(\bm y,s)>0\quad\mbox{a.e.\! }(\bm y,s)\in E,\ \mbox{a.e.\! }\bm x\in\Om,
\]
indicating
\[
u(\bm x,t)\ge\int_E G_\al(\bm x,\bm y,t-s)F(\bm y,s)\,\rd\bm y\rd s>0,\quad\mbox{a.e.\! }\bm x\in\Om.
\]
This completes the proof.
\end{proof}

Although the strict positivity for the homogeneous problem has been established in \cite{LRY16,Z13}, here we provide an alternative proof. As before, we first prove the strict positivity of the Green's function associated with the fractional homogeneous problem, and then use it to establish the strict positivity of the solution.

\begin{theorem}\label{thm:2}
Let $0<\al\le 1$ and let $K_\al(\bm x,\bm y,t)$ be the Green's function associated with the homogeneous problem defined by
\[
K_\al(\bm x,\bm y,t)=\sum_{n=1}^\infty E_{\al,1}(-\la_n t^\al)\vp_n(\bm x)\vp_n(\bm y),\quad\mbox{a.e.\! }\bm x,\bm y\in\Om,\ t\in(0,T).
\]
Then $K_\al> 0$ a.e.\! in $\Om\times\Om\times(0,T)$.
\end{theorem}

\begin{proof}
First of all, again we invoke the strictly positive of Green's function for the heat equation with homogeneous Dirichlet boundary conditions, i.e.,
\[
K_1(\bm x,\bm y,t)=\sum_{n=1}^\infty\e^{-\la_n t}\vp_n(\bm x)\vp_n(\bm y)>0, \quad\mbox{a.e.\! }\bm x,\bm y\in\Om,\ t\in(0,T).
\]
Similarly, for homogeneous problem, we define $h^\al_n(t):=E_{\al,1}(-\la_n t^\al)$ so that Green's function for the subdiffusion equation can be written as
\[K_\al(\bm x,\bm y,t) :=\sum_{n=1}^\infty h^\al_n(t)\vp_n(\bm x)\vp_n(\bm y).\]
By the Laplace transform formula for the Mittag-Leffler function, we have
\[
\cL h^\al_n(s)=\f{s^{\al-1}}{s^\al+\la_n}=\int_0^\infty s^{\al-1}\e^{-(s^\al+\la_n)\tau}\,\rd\tau.
\]
As before, we apply the inverse Laplace transform and represent $h^\al_n(t)$ via the Wright function
\begin{align*}
h^\al_n(t)&=\cL^{-1}\left(\int_0^\infty s^{\al-1}\e^{-(s^\al+\la_n)\tau}\,\rd\tau\right)(t)=\int_0^\infty\e^{-\la_n\tau}\cL^{-1}\left(s^{\al-1}\e^{-s^\al\tau}\right)(t)\,\rd\tau\\
& =\int_0^\infty\e^{-\la_n\tau}\phi_\al(t,\tau)\,\rd\tau,
\end{align*}
where we defined $\phi_\al(t,\tau):=\cL^{-1}(s^{\al-1}\e^{-s^\al\tau})(t)$. Then we recall the Laplace inverse transform of $s^{\al-1}\e^{-r s^\al}$ (see \cite[p.~247, Eq.~(19)]{L19})
\[
\cL^{-1}\left(s^{\al-1}\e^{-r s^\al}\right)(t)=\f1{t^\al}M_\al(r t^{-\al}),\quad r>0.
\]
Again, it follows from Lemma~\ref{lem:M-nonnegtive} that
\[
\phi_\al(t,\tau)=t^{-\al}M_\al(\tau t^{-\al})\ge0,
\]
indicating $h^\al_n(t)>0$. Last but not least, substituting $h_n^\al(t)$ back into $K_\al(\bm x,\bm y,t)$ yields
\begin{align*}
K_\al(\bm x,\bm y,t) & =\sum_{n=1}^\infty\int_0^\infty\e^{-\la_n\tau}\phi_\al(t,\tau)\,\rd\tau\,\vp_n(\bm x)\vp_n(\bm y)\\
& =\int_0^\infty\sum_{n=1}^\infty\e^{-\la_n\tau}\vp_n(\bm x)\vp_n(\bm y)\phi_\al(t,\tau)\,\rd\tau\\
& =\int_0^\infty K_1(\bm x,\bm y,\tau)\phi_\al(t,\tau)\,\rd\tau.
\end{align*}
Here, since $K_1(\bm x,\bm y,\tau)>0$ and $\phi_\al(t,\tau)\ge0,\not\equiv0$, again Tonelli's theorem enables the interchange of the order of integration. This complete the proof.
\end{proof}

We now exploit the strict positivity of Green's function for the homogeneous problem established above to prove the strict positivity of the solution.

\begin{theorem}\label{thm:SMP(F=0)}
Let $F=0$ and $a\in D(\cA^\be)$ with $\be>d/4-1$. If $a\ge0,\not\equiv0$ in $\Om,$ then the solution $u$ to \eqref{Eq:IBVP} satisfies $u>0$ in $\Om\times(0,T)$.
\end{theorem}

\begin{proof}
By the definition of $K_\al$ and Lemma~\ref{Lem:scalar_well-posedness}(i), the solution to \eqref{Eq:IBVP} admits the representation
\[
u(\bm x,t)=\int_\Om K_\al(\bm x,\bm y,t)a(\bm y)\,\rd \bm y.
\]
From Theorem~\ref{thm:2}, we have $K_\al>0$ a.e.\! in $\Om\times\Om\times(0,T)$. Hence, it follows that
\[
u(\bm x,t)>0,\quad \bm x\in\Om,\ t>0.
\]
This completes the proof.
\end{proof}

As a direct combination of Theorem~\ref{thm:SMP(a=0)} and Theorem~\ref{thm:SMP(F=0)}, we obtain the following corollary which completes this section by addressing the general case.

\begin{corollary}\label{coro:strict-positive}
Let $a\in D(\cA^\be)$ and $F\in L^2(0,T;D(\cA^\be))$ with $\be>d/4-1$. Assume $a\ge0,F\ge0$ and at least one of them does not vanish identically. Define
\[
t_*:=\begin{cases}
0,\quad & \mbox{if }a\not\equiv0,\\
\inf\{t\ge0\mid F(\,\cdot\,,t)\not\equiv0\mbox{ in }\Om\}, & \mbox{otherwise},
\end{cases}
\]
which represents the first activation time of the equation.
Then the solution $u$ to \eqref{Eq:IBVP} satisfies $u > 0$ in $\Om\times (t_*,T)$.
\end{corollary}

\begin{proof}
By the linearity of the governing equation and superposition principle, the solution to \eqref{Eq:IBVP} can be decomposed into $u=u_1+u_2$ where $u_1,u_2$ are the solutions to the homogeneous and inhomogeneous problems, respectively. Since each part is strictly positive under the non-negativity assumption of Theorem \ref{thm:SMP(a=0)} and Theorem \ref{thm:SMP(F=0)}, their sum remains strictly positive.
\end{proof}


\section{Strict Positivity for Coupled System}\label{sec-couple}

We now generalize the strict positivity properties obtained for single equations in the above section to coupled systems \eqref{Eq:Coupled_prototype_matrix}.

To address Problem~\ref{prob:coupled} in a systematic manner, the central task is to elucidate the role of coupling in the propagation of strict positivity. To this end, we employ Theorem~\ref{thm:coupled_F=0} and Theorem~\ref{thm:time-propagation} to resolve these two problems respectively. First and foremost, a clear understanding of these results requires an assumption on the connectivity of the coupled system. This is essential because in the coupled system under our consideration not every equation is equipped with a nonzero source term or nontrivial initial data. For any equation whose initial data and source term both vanish identically, the corresponding solution is always trivial. However, the key point of interest lies precisely in such equations: {\it if there exists a directed path from an equation with a nonzero solution to one of these trivial equations, then the latter will also become nontrivial under the influence of the coupling}. An even more intriguing question concerns {\it the time at which such activation occurs}. For the homogeneous problem in cooperative coupled systems, there exists a unified “activation time”, whereas for the inhomogeneous problem the activation times may differ across components.

To clarify these phenomena, we begin with Proposition~\ref{prop:coupled} where we establish a connection between the matrix formulation and the index-set representation of connectivity thereby capturing the underlying interaction structure.

\begin{proposition}\label{prop:coupled}
Let $\bm a$ be the initial value and $\bm C$ the coefficient matrix in \eqref{Eq:Coupled_prototype_matrix}. Define a vector $\bm r=(r_1,\dots,r_K)^\T\in \BR^K$ and a matrix $\bm Q=(q_{k\ell})_{1\le k,\ell\le K}\in\BR^{K\times K}$ by
\[
r_k=\begin{cases}
1, & \mbox{if }a_k\not\equiv0\mbox{ in }\Om,\\
0, & \mbox{otherwise},
\end{cases}\quad q_{kk}=1,\quad q_{k\ell}=\begin{cases}
1, & \mbox{if }c_{k\ell}\not\equiv0\mbox{ in }\Om,\\
0, & \mbox{otherwise},
\end{cases}
\quad k\ne\ell.
\]
Define the influence sets by
\begin{equation}\label{Eq:prop_Im}
I_m:=\begin{cases}
\{k\mid a_k\not\equiv0\mbox{ in }\Om\}, & m=0,\\
I_{m-1}\cup\{k\mid\exists\,\ell\in I_{m-1}\mbox{ such that }c_{k\ell}\not\equiv0\mbox{ in }\Om\}, & m=1,2,\dots.
\end{cases}
\end{equation}
Then for all $m=0,1,\dots$ and $k=1,\dots,K,$ we have $k\in I_m$ if and only if $(\bm Q^m\bm r)_k>0$.
\end{proposition}

\begin{remark}\label{rem:prop}
\begin{enumerate}
\item We present two equivalent ways to describe the connectivity induced by nonzero initial data. For clarity, we establish the equivalence in the case of nonzero initial data while the argument for source terms at a fixed time is analogous.
\item The necessity of Proposition~\ref{prop:coupled} lies in its role in Theorem~\ref{thm:time-propagation}. In that theorem, the index-set formulation is essential for characterizing the influence region of a single nonzero component. This characterization enables us to analyze the time propagation of strict positivity in the coupled system by means of Picard iteration and the superposition principle. 
\item We interpret $I_0$ in \eqref{Eq:prop_Im} as the index set of components with nonzero initial data, i.e., it contains all indexes corresponding to nontrivial initial conditions in the coupled system. For $m=1,2,\dots$, the set $I_m$ defined by \eqref{Eq:prop_Im} consists of all indexes that can be reached through nonzero coupling terms at the $m$-th step. From a propagation perspective, this means at step $m$ each index $\ell$ transmits its “influence” to all indexes $k$ connected to it through nonzero coupling. Therefore, \eqref{Eq:prop_Im} characterizes the set of components activated by the previous step.
\end{enumerate}
\end{remark}

\begin{example}
Before proceeding to the prove of Proposition~\ref{prop:coupled}, we first use Figure~\ref{fig-graph} to illustrate the notations and the cooperative condition \eqref{eq-coop}, as well as to explain how the Picard sequence is employed at each step of the induction. 

Let $K=5$ and assume that the initial condition vector $\bm a$ and the coupling matrix $\bm C$ take forms of
\[
    \bm a=\begin{bmatrix}
        0\\
        \bullet\\
        0\\
        0\\
        0
    \end{bmatrix},\quad 
    \bm C=\begin{bmatrix}
    \times & *      & 0      & *      & 0 \\
    *      & \times & 0      & *      & 0 \\
    *      & 0      & \times & 0      & 0 \\
    0      & *      & *      & \times & * \\
    0      & 0      & *      & 0      & \times
    \end{bmatrix}.
\]
Here $\bullet$ denotes a nontrivial component in the initial value, and $\times$ represents an arbitrary (trivial or nontrivial) diagonal entry in $\bm C$, since no specific assumptions on $c_{kk}$ ($k=1,\dots,K$). Meanwhile, $*$ indicates a non-positive and nontrivial coefficient.

On the other hand, according to the definitions in Proposition~\ref{prop:coupled}, the vector $\bm r$ and the matrix $\bm Q$ are defined as
\[
    \bm r=\begin{bmatrix}
        0\\
        1\\
        0\\
        0\\
        0
    \end{bmatrix}, \quad 
    \bm Q=\begin{bmatrix}
    1 & 1 & 0 & 1 & 0 \\
    1 & 1 & 0 & 1 & 0 \\
    1 & 0 & 1 & 0 & 0 \\
    0 & 1 & 1 & 1 & 1 \\
    0 & 0 & 1 & 0 & 1
    \end{bmatrix}.
\]
Further, the corresponding influence sets $I_k$ at each iteration step are given as follows:
\[
I_0=\{2\},\quad I_1=\{1,2,4\},\quad I_2=\{1,2,3,4\},\quad I_3=\{1,2,3,4,5\}.
\]
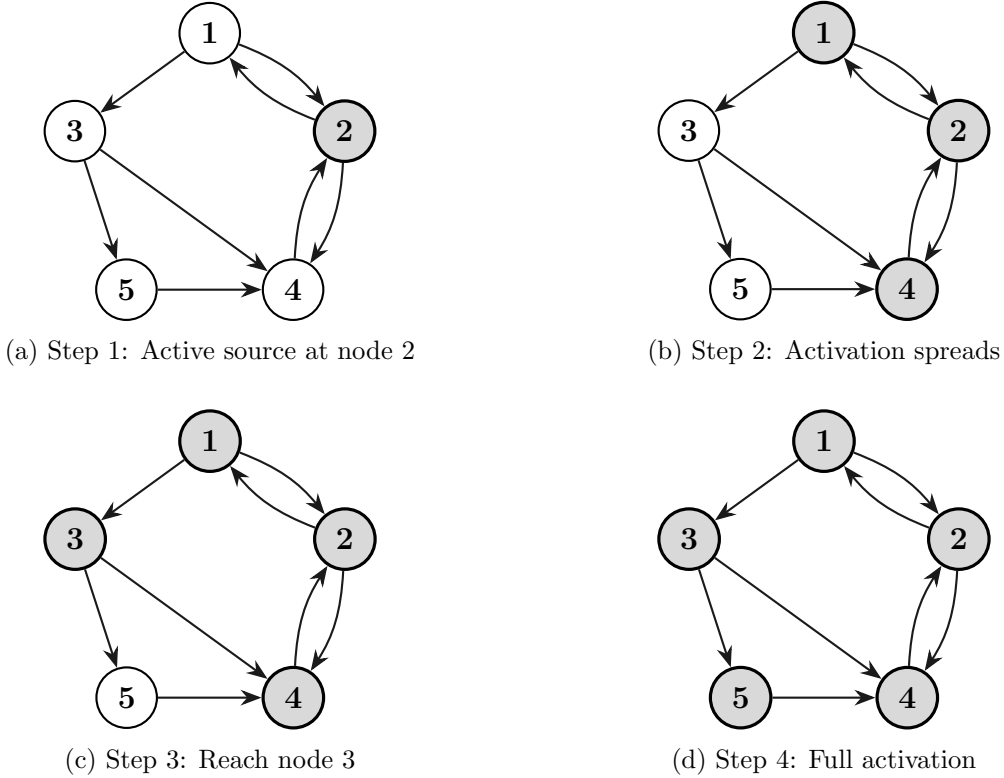
\begin{figure}[htbp]
\centering
\begin{subfigure}[b]{0.49\textwidth}
\centering
\begin{tikzpicture}[scale=0.75]
\node[base] (N1) at (90:2.5) {1};
\node[active_node] (N2) at (18:2.5) {2};
\node[base] (N4) at (306:2.5) {4};
\node[base] (N5) at (234:2.5) {5};
\node[base] (N3) at (162:2.5) {3};
\draw[arrow] (N1) to[bend left=15] (N2); \draw[arrow] (N2) to[bend left=15] (N1);
\draw[arrow] (N2) to[bend left=15] (N4); \draw[arrow] (N4) to[bend left=15] (N2);
\draw[arrow] (N1) to (N3); \draw[arrow] (N3) to (N5);
\draw[arrow] (N5) to (N4); \draw[arrow] (N3) to (N4);
\end{tikzpicture}
\caption{Step 1: Active source at node 2}
\end{subfigure}
\hfill
\begin{subfigure}[b]{0.49\textwidth}
\centering
\begin{tikzpicture}[scale=0.75]
\node[active_node] (N1) at (90:2.5) {1};
\node[active_node] (N2) at (18:2.5) {2};
\node[active_node] (N4) at (306:2.5) {4};
\node[base] (N5) at (234:2.5) {5};
\node[base] (N3) at (162:2.5) {3};
\draw[arrow] (N1) to[bend left=15] (N2); \draw[arrow] (N2) to[bend left=15] (N1);
\draw[arrow] (N2) to[bend left=15] (N4); \draw[arrow] (N4) to[bend left=15] (N2);
\draw[arrow] (N1) to (N3); \draw[arrow] (N3) to (N5);
\draw[arrow] (N5) to (N4); \draw[arrow] (N3) to (N4);
\end{tikzpicture}
\caption{Step 2: Activation spreads}
\end{subfigure}\vspace{5mm}
\begin{subfigure}[b]{0.49\textwidth}
\centering
\begin{tikzpicture}[scale=0.75]
\node[active_node] (N1) at (90:2.5) {1};
\node[active_node] (N2) at (18:2.5) {2};
\node[active_node] (N4) at (306:2.5) {4};
\node[base] (N5) at (234:2.5) {5};
\node[active_node] (N3) at (162:2.5) {3};
\draw[arrow] (N1) to[bend left=15] (N2); \draw[arrow] (N2) to[bend left=15] (N1);
\draw[arrow] (N2) to[bend left=15] (N4); \draw[arrow] (N4) to[bend left=15] (N2);
\draw[arrow] (N1) to (N3); \draw[arrow] (N3) to (N5);
\draw[arrow] (N5) to (N4); \draw[arrow] (N3) to (N4);
\end{tikzpicture}
\caption{Step 3: Reach node 3}
\end{subfigure}
\hfill
\begin{subfigure}[b]{0.49\textwidth}
\centering
\begin{tikzpicture}[scale=0.75]
\node[active_node] (N1) at (90:2.5) {1};
\node[active_node] (N2) at (18:2.5) {2};
\node[active_node] (N4) at (306:2.5) {4};
\node[active_node] (N5) at (234:2.5) {5};
\node[active_node] (N3) at (162:2.5) {3};
\draw[arrow] (N1) to[bend left=15] (N2); \draw[arrow] (N2) to[bend left=15] (N1);
\draw[arrow] (N2) to[bend left=15] (N4); \draw[arrow] (N4) to[bend left=15] (N2);
\draw[arrow] (N1) to (N3); \draw[arrow] (N3) to (N5);
\draw[arrow] (N5) to (N4); \draw[arrow] (N3) to (N4);
\end{tikzpicture}
\caption{Step 4: Full activation}
\end{subfigure}
\caption{Four stages of the connectivity proof.}\label{fig-graph}
\end{figure}
\end{example}

\begin{proof}
We proceed by induction on $m$. For $m=0$, it follows immediately from the definition of $\bm r$ that $r_k=1$ if and only if $k\in I_0$.

Now assume that for some $m=0,1,\dots$, $k\in I_m$ if and only if $(\bm Q^m\bm r)_k>0$. We prove the statement for $m+1$. Suppose that $k\in I_{m+1}$, that is, there exists $\ell\in I_m$ such that $c_{k\ell}\not\equiv0$. By the definition of $\bm Q$, it is equivalent to $q_{k\ell}=1$. Then by the induction hypothesis, we see $(\bm Q^m\bm r)_\ell>0$ and hence
\begin{equation}\label{eq-Qrk}
(\bm Q^{m+1}\bm r)_k=\sum_{\ell=1}^K q_{k\ell}(\bm Q^m\bm r)_\ell>0.
\end{equation}
Conversely, suppose that \eqref{eq-Qrk} holds true for some $k$. Since each term in the sum is nonnegative, there exists $\ell$ such that
\[
q_{k\ell}(\bm Q^m\bm r)_\ell>0.
\]
Then there should hold $q_{k\ell}=1$ and $(\bm Q^m\bm r)_\ell>0$. By the induction hypothesis, we have $\ell\in I_m$. Moreover, $q_{k\ell}=1$ is equivalent to $c_{k\ell}\not\equiv0$ and thus $k\in I_{m+1}$.
\end{proof}

With the above preparations, we are now ready to establish the propagation of positivity in the coupled system. We begin by considering the case of homogeneous problem.

\begin{theorem}\label{thm:coupled_F=0}
Let $\bm F=\bm0,$ $\bm a\in D(\cA^\be)$ with $\be>d/4-1,$ and $\bm C\in L^\infty(\Om)$ satisfy the cooperativeness condition \eqref{eq-coop}. Let $I_m$ be the influence sets defined in \eqref{Eq:prop_Im} and assume
\begin{equation}\label{Eq:connection_3}
\lim_{m\to\infty}I_m=\{1,\dots,K\}.
\end{equation}
If $\bm a\ge\bm0,\not\equiv\bm0$ in $\Om,$ then the solution $\bm u$ to \eqref{Eq:Coupled_prototype_matrix} satisfies $\bm u>\bm0$ in $\Om\times(0,T)$.
\end{theorem}

\begin{remark}\label{Rem:thm_coupled_F=0}
The assumptions in the above theorem can be divided into two categories. First, the condition \eqref{eq-coop} requires that all inter-component interactions are nonnegative excluding self-effects, thereby ensuring the cooperative nature of the system. Next, \eqref{Eq:connection_3} requires the connectivity throughout all components of the coupled system. In fact, the introduction of influence sets \eqref{Eq:prop_Im} was explained in Remark~\ref{rem:prop}, and condition \eqref{Eq:connection_3} further ensures that this propagation reaches the entire system in finite steps.
\end{remark}

\begin{proof}
The proof of this theorem is divided into two steps. In the first step, we follow the proof of \cite[Theorem 2]{BL26} to construct a specified Picard iteration sequence converging to the mild solution to \eqref{Eq:Coupled_prototype_matrix} and establish its monotonicity. On the basis of these existing results, we prove that the Picard sequence is non-decreasing. In the second step, we apply Theorem~\ref{thm:SMP(a=0)} to conclude the strict positivity of the solution.

First of all, we rewrite \eqref{Eq:Coupled_prototype_matrix} into a slightly different form. Let
\[
\bm\Sigma:=\diag(\|c_{11}\|_{L^\infty(\Om)},\dots,\|c_{KK}\|_{L^\infty(\Om)}),\quad\bm B:=\bm\Sigma-\bm C=\begin{cases}
-c_{k\ell}, & k\ne\ell,\\
\|c_{kk}\|_{L^\infty(\Om)}-c_{kk}, & k=\ell,
\end{cases}
\]
and define
\[
\wt\cA_k:=\cA_k+\|c_{kk}\|_{L^\infty(\Om)},\quad\wt\BA:=\BA+\bm\Sigma=\diag\left(\wt\cA_1,\dots,\wt\cA_K\right),\quad k=1,\dots,K.
\]
Owing to the cooperativeness condition \eqref{eq-coop} on $\bm C$, we know that $\bm B$ is a nonnegative matrix, i.e., $\bm B\ge\bm O$. Here we emphasize that in the sequel, since the propagation at each step proceeds from some $\ell\in I_m$ to $k\in I_{m+1}$ according to \eqref{Eq:prop_Im}, the use of $\bm B$ instead of $-\bm C$ does not affect the induction argument. Now we can reformulate \eqref{Eq:Coupled_prototype_matrix} as
\[
\begin{cases}
\pa_t^{\bm\al}(\bm u-\bm a)+\wt\BA\bm u=\bm B\bm u\quad & \mbox{in }\Om\times(0,T),\\
\bm u=\bm0 & \mbox{on }\pa\Om\times(0,T),
\end{cases}
\]
where $\wt\BA$ shares the same spectral property as that of $\BA$. 

Next, we construct a sequence $\{\bm u^{(m)}\}_{m=0}^\infty$ convergent to $\bm u$ by $\bm u^{(0)}=\bm0$ and
\begin{equation}\label{Eq:process1}
\begin{cases}
\pa_t^{\bm\al}(\bm u^{(m)}-\bm a)+\wt\BA\bm u^{(m)}=\bm B\bm u^{(m-1)} & \mbox{in }\Om\times(0,T),\\
\bm u^{(m)}=\bm0 & \mbox{on }\pa\Om\times(0,T),
\end{cases}\quad m=1,2,\dots.
\end{equation}
The convergence of this sequence to the unique mild solution of \eqref{Eq:Coupled_prototype_matrix} is guaranteed by \cite{LHL23}, and the proof of \cite[Proposition 2]{BL26} implies that $\{\bm u^{(m)}\}$ is a non-decreasing sequence. Especially, we know that the mild solution $\bm u\ge\bm u^{(m)}$ for any $m\in\BN$.

Now we aim at showing by induction on $m$ that
\begin{equation}\label{eq-assert}
\forall\,m=0,1,\dots,\ \forall\,k\in I_m,\quad u_k^{(m+1)}>0\quad\mbox{in }\Om\times(0,T).
\end{equation}
For any $k\in I_0$, we have $a_k\ge0,\not\equiv0$ in $\Om$ by definition, and $u_k^{(1)}$ satisfies
\[
\begin{cases}
\pa_t^{\al_k}(u_k^{(1)}-a_k)+\wt \cA_k u_k^{(1)}=0 & \mbox{in }\Om\times(0,T),\\
u_k^{(1)}=0 & \mbox{on }\pa\Om\times(0,T).
\end{cases}
\]
Then it follows directly from Theorem \ref{thm:SMP(F=0)} that $u_k^{(1)}>0$ in $\Om\times(0,T)$.

Now we assume that the assertion \eqref{eq-assert} holds true for some $m-1=0,1,\dots$. For any $k\in I_m$, we divide into two cases, namely $k\in I_m\cap I_{m-1}$ and $k\in I_m\setminus I_{m-1}$. In the case of $k\in I_m\cap I_{m-1}$, it follows from the induction assumption that $u_k^{(m)}>0$ in $\Om\times(0,T)$. Then the monotonicity of $\{\bm u^{(m)}\}$ indicates $u_k^{(m+1)}\ge u_k^{(m)}>0$ in $\Om\times(0,T)$.

In the case of $k\in I_m\setminus I_{m-1}$, first we know from definition \eqref{Eq:prop_Im} that there exists $\ell\in I_{m-1}$ such that $c_{k\ell}\not\equiv0$ in $\Om$. For such an $\ell$, the induction assumption asserts $u_\ell^{(m)}>0$ in $\Om\times(0,T)$. On the other hand, since $k\notin I_{m-1}$, we have $k\ne\ell$ and hence $c_{k\ell}\le0,\not\equiv0$ in $\Om$ according to \eqref{eq-coop}. Then by the definition of $\bm B$, we obtain
\[
(\bm B\bm u^{(m)})_k=\sum_{\ell=1}^K b_{k\ell}u_\ell^{(m)}\ge-c_{k\ell}u_\ell^{(m)}\ge0,\not\equiv0\quad\mbox{in }\Om\times(0,T),
\]
and $\inf\{t\ge0\mid(\bm B\bm u^{(m)})_k\not\equiv0\mbox{ in }\Om\}=0$. Consequently, Theorem \ref{thm:SMP(a=0)} claims that the solution to the inhomogeneous problem
\[
\begin{cases}
(\pa_t^{\al_k}+\wt \cA_k)u_k^{(m+1)}=(\bm B\bm u^{(m)})_k & \mbox{in }\Om\times(0,T),\\
u_k^{(m+1)}=0 & \mbox{on }\pa\Om\times(0,T)
\end{cases}
\]
satisfies $u_k^{(m+1)}>0$ in $\Om\times(0,T)$. This completes the verification of assertion \eqref{eq-assert}.

Finally, thanks to the connectivity condition \eqref{Eq:connection_3}, each index $k$ belongs to $I_{m_k}$ for some $m_k=0,1,\dots$. By the induction result above, it follows that
\[
u_k^{(m_k+1)}>0\quad\mbox{in }\Om\times(0,T),\ \forall\,k=1,\dots,K.
\]
Setting $M:=\max_{1\le k\le K}m_k+1$ and employing the monotone convergence of $\{\bm u^{(m)}\}$ to the mild solution $\bm u$, we conclude
\[
\bm u\ge\bm u^{(M)}>\bm0\quad\mbox{in }\Om\times(0,T),
\]
which completes the proof.
\end{proof}

Finally, we present the following theorem to answer the inhomogeneous part of Problem~\ref{prob:coupled} which concerns the time-propagation behavior induced by source terms. Before stating the result, we briefly highlight an essential difference between the homogeneous problem and the inhomogeneous problem which directly influences the connectivity assumptions in the theorem.

For the homogeneous problem, the activation occurs simultaneously at $t=0$. Therefore, the definition of $I_0$ does not depend on a specific starting component and $I_0$ need not be a singleton. Under assumption \eqref{Eq:connection_3}, each equation is connected to at least one component with nonvanishing initial data thus receives the propagated influence. Since all components share the same activation time namely $t=0$, the entire system becomes nontrivial from the initial moment onward.

In contrast, the inhomogeneous problem is substantially more intricate. Each component may possess a source term that becomes active at a different time. For instance, the $k$-th equation may have a source term that becomes nontrivial only after some time $t_k$ while it is coupled with another equation $l$ whose source term is activated at a different time $t_l$, and so on. The interaction of such time-dependent sources across the coupled system leads to a more complex propagation mechanism.

To handle this difficulty, we adopt a decomposition strategy based on the superposition principle. Specifically, the original system is decomposed into $K$ subsystems, each of which contains a single source term. For each subsystem associated with a given index $l$, we impose a corresponding connectivity assumption meaning that the index set $I_0$ is defined with respect to that particular source. From the perspective of a receiving component $k$, its positivity is influenced by all possible source components $l$ to which it is connected. The time at which $u_k$ becomes strictly positive is then determined by the earliest propagation among all such paths.

This idea is rigorously characterized by the following theorem.

\begin{theorem}\label{thm:time-propagation}
Let $\bm a=\bm0,$ $\bm F\in L^2(0,T;D(\cA^\be))$ with $\be>d/4-1,$ and $\bm C\in L^\infty(\Om)$ satisfy the cooperativeness condition \eqref{eq-coop}. Assume $\bm F\ge\bm0,\not\equiv\bm0$ in $\Om\times(0,T)$ and for each $k=1,\dots,K,$ define the activation time by
\begin{equation}\label{assump:2_1}
t_k:=\inf\{t\ge0\mid F_k(\,\cdot\,,t)\not\equiv0\mbox{ in }\Om\}.
\end{equation}
For each $\ell=1,\dots,K,$ further define the reachable sets by
\begin{equation}\label{assump:2_2}
\begin{aligned}
I_0^\ell & :=\begin{cases}
\{\ell\}, & t_\ell<\infty,\\
\emptyset, & t_\ell=\infty,
\end{cases}\\
I_m^\ell & :=\{
k\mid\exists\,j\in I_{m-1}^\ell\mbox{ such that }c_{kj}\not\equiv0\mbox{ in }\Om\},\quad m=1,2,\dots.
\end{aligned}
\end{equation}
Moreover, we define
\begin{equation}\label{assump:2_3}
R_\ell:=\lim_{m\to\infty}I_m^\ell,\quad\ell=1,\dots,K
\end{equation}
and assume
\begin{equation}\label{assump:2_4}
\bigcup_{\ell=1}^K R_\ell=\{1,\dots,K\}.
\end{equation}
Then for each $k=1,\dots,K,$ the solution to \eqref{Eq:Coupled_prototype_matrix} satisfies
\begin{equation}\label{assump:3_1}
u_k>0\quad\mbox{in }\Om\times(\tau_k,T),\quad\tau_k:=\min\{t_\ell\mid\ell=1,\dots,K,\ k\in R_\ell\}.
\end{equation}
\end{theorem}

\begin{remark} We now explain the assumptions \eqref{assump:2_1}--\eqref{assump:2_4} and the corresponding conclusion \eqref{assump:3_1}.
\begin{enumerate}
\item Assumptions \eqref{assump:2_1}--\eqref{assump:2_4} concerns the connectivity of the coupled system and consists of several components. The first condition \eqref{assump:2_1} specifies the activation time of the source term in each equation which also allows the possibility that some equations admit no source term, i.e., $t_k=\infty$. The conditions \eqref{assump:2_2}--\eqref{assump:2_3} describe the connectivity among equations with nontrivial source terms. In particular, for each such equation $\ell$, we define $R_\ell$ as the set of equations that can be reached from $\ell$ through finitely many steps, namely, the reachable set starting from $\ell$. Finally, condition \eqref{assump:2_4} imposes a global connectivity requirement, ensuring that the entire coupled system is connected to the source terms.
\item The conclusion \eqref{assump:3_1} indicates that different components may become strictly positive at different moments, reflecting a propagation of positivity. More precisely, for each equation the time at which it becomes positive coincides with the earliest time at which any source term that can influence it (including its own) becomes positive.
\end{enumerate}
\end{remark}

\begin{proof}
We divide the proof into two steps. In the first step, we establish the propagation of positivity from a single source. In the second step, we extend the result to multiple sources via the superposition principle.

We begin by using the Picard iteration with trivial initial condition to establish the propagation of positivity from a nontrivial source term. The convergence of the sequence \eqref{Eq:process2} to the unique mild solution and its monotonicity are guaranteed by \cite{LHL23} and \cite[Proposition 2]{BL26}, respectively.
\begin{equation}\label{Eq:process2}
\begin{cases}
(\pa_t^{\bm\al}+\wt\BA)\bm u^{(m)}=\bm B\bm u^{(m-1)}+\bm F & \mbox{in }\Om\times(0,T),\\
\bm u^{(m)}=\bm0 & \mbox{on }\pa\Om\times(0,T),
\end{cases}\quad m=1,2,\dots.
\end{equation}

Since the case $t_\ell=\infty$ is trivial, we fix $\ell\in\{1,\dots,K\}$ with $t_\ell<\infty$. Recall the definition of the index sets of $\ell$ in assumption \eqref{assump:2_2}, we now aim at showing by induction on $m$ that
\begin{equation}\label{assertion:2}
    \forall\,m=0,1,\dots,\ \forall k\in I_m^\ell,\quad u_k^{(m+1)}>0\quad\mbox{in }\Om\times(t_\ell,T),
\end{equation}
where $t_\ell$ was defined in \eqref{assump:2_1}.

For the initial step, $I_0^\ell=\{\ell\}$, $t_\ell=\inf\{t\ge0\mid F_\ell(\,\cdot\,,t)\not\equiv0\mbox{ in }\Om\}<\infty$ and $u_\ell^{(1)}$ satisfies
\[
\begin{cases}
\pa_t^{\al_\ell}u_\ell^{(1)}+\wt \cA_\ell u_\ell^{(1)}=F_\ell & \mbox{in }\Om\times(0,T),\\
u_\ell^{(1)}=0 & \mbox{on }\pa\Om\times(0,T).
\end{cases}
\]
Since $F_\ell\ge0$ and $F_\ell\not\equiv0$, Theorem~\ref{thm:SMP(a=0)} yields 
$u_\ell^{(1)}>0 \;\mbox{in }\Om\times(t_\ell,T).$

Now we assume that the assertion \eqref{assertion:2} holds true for some
$m-1=0,1,\dots$. For any $k\in I_m^\ell$, we distinguish two cases, $k\in I_m^\ell\cap I_{m-1}^\ell$ and $k\in I_m^\ell\setminus I_{m-1}^\ell$. For the case of $k\in I_m^\ell\cap I_{m-1}^\ell$, then the induction hypothesis gives $k\in I_{m-1}^\ell$, that is $u_k^{(m)}>0 $ in $\Om\times(t_\ell,T).$ By the monotonicity of the Picard sequence $\bm u^{(m)}$, we have $u_k^{(m+1)}\ge u_k^{(m)}>0$ in $\Om\times(t_\ell,T).$

For the case of $k\in I_m^\ell\setminus I_{m-1}^\ell$, it follow from the definition~\eqref{assump:2_2} that there exists $j\in I_{m-1}^\ell$ such that $c_{kj}\not\equiv0$ in $\Om$. For such a $j$, the induction assumption asserts $u_j^{(m)}>0$ in $\Om\times(t_\ell,T).$ On the other hand, since $k\not\in I_{m-1}^\ell$, we have $k\not=j$. Hence $c_{kj}\le0,\not\equiv0$ in $\Om$ according to \eqref{eq-coop}. Then by the definition of $\bm B$, we obtain
\[
(\bm B\bm u^{(m)})_k=\sum_{i=1}^Kb_{ki}u_i^{(m)}\ge-c_{kj}u_j^{(m)}\ge0,\not\equiv0, \quad\mbox{in }\Om\times(0,T),
\]
and $\inf\{t\ge0\mid(\bm B\bm u^{(m)})_k\not\equiv0 \mbox{ in }\Om\}=t_\ell$.

Consequently combining with the non-negativity of $\bm F$, then by Theorem~\ref{thm:SMP(a=0)} claims that the solution to the inhomogeneous problem
\[
\begin{cases}
\pa_t^{\al_k}u_k^{(m+1)}+\wt \cA_k u_k^{(m+1)}=(\bm B\bm u^{(m)})_k+F_k & \mbox{in }\Om\times(0,T),\\
u_k^{(m+1)}=0 & \mbox{on }\pa\Om\times(0,T),
\end{cases}
\]
satisfies $u_k^{(m+1)}>0$ at least in $\Om\times(t_\ell,T)$. This completes the verification of assertion~\eqref{assertion:2}.

Finally, according to the connectivity condition~\eqref{assump:2_4}, each index $k$ belongs to $R_\ell$ which has been defined in \eqref{assump:2_3}, that is, every $k\in R_\ell$ belongs to $I_{m_k}^\ell$ for some $m_k=0,1,\dots,$ and by the induction result~\eqref{assertion:2}, it follows that
\begin{equation}\label{thm:induction_for_k_inCoupled}
\forall\,k\in R_\ell,\quad u_k>0\quad\mbox{in }\Om\times(t_\ell,T).
\end{equation}

In the second step, we take advantage of the superposition principle to investigate the auxiliary problem driven only by the $\ell$-th component of the source, i.e., $\bm F=\sum_{{\ell=1}}^K F_\ell\bm e_\ell$, where $\{\bm e_1,\dots,\bm e_K\}$ denotes the standard basis of $\BR^K$ and let $\bm u_\ell$ be the corresponding solution. Due to the linearity of the system, it follows that $\bm u=\sum_{\ell=1}^K\bm u_\ell$ or equivalently, for $k=1,\dots,K$, $u_k=\sum_{\ell=1}^K (\bm u_\ell)_k$. Furthermore, note that for each $\ell$ the contribution $u_k^{(\ell)}$ is nontrivial only if $k\in R_\ell$. Hence, we arrive at
\[
u_k=\sum_{\ell:\, k\in R_\ell} (\bm u_\ell)_k.
\]

For each such $\ell$, we recall the conclusion obtained in \eqref{thm:induction_for_k_inCoupled} to define
\[
\tau_k:=\min\{t_\ell\mid\ell=1,\dots,K,\ k\in R_\ell\}.
\]
For $t>\tau_k$, there exists $\ell$ such that $k\in R_\ell$ and $t>t_\ell$. Therefore, we arrive at $(\bm u_\ell)_k>0$ in $\Om\times(\tau_k,T)$, and eventually
\[
u_k=\sum_{\ell:k\in R_\ell}(\bm u_\ell)_k\ge(\bm u_\ell)_k>0\quad\mbox{in }\Om\times(\tau_k,T).
\]
This completes the proof.
\end{proof}


\section{Conclusions}\label{sec-concl}

The strict positivity of solutions to time-fractional diffusion equations is a fundamental property. However for coupled systems, a systematic understanding of positivity propagation was absent. In this work, we propose a new perspective for analyzing the positivity mechanism of subdiffusion equations and systems. For scalar equations, we establish a connection between subdiffusion equations and classical heat equations, allowing the former to inherit the positivity of the latter. Based on this idea, we prove the strict positivity of solutions for both homogeneous and inhomogeneous problems.

Building upon these results, we further investigate the propagation of positivity in coupled systems. In particular, we characterize how positivity evolves over time under both initial conditions and source terms, and reveal the interaction-driven propagation mechanism among different components. This study provides a structural understanding of coupled systems and lays a foundation for the analysis of more complex problems, such as nonlinear coupled systems and degenerate parabolic equations. Moreover, the proposed framework has potential applications in a variety of models, including population dynamics, predator-prey systems, tumor growth models and the propagation of activation signals in neural networks.


\paragraph{Acknowledgment}

The first author is supported by China Scholarship Council and National Natural Science Foundation of China (Nos.\! 12371428 and 11871435). The second author is supported by JSPS KAKENHI Grant Numbers JP23KK0049 and JP26K06926, Guangdong Basic and Applied Basic Research Foundation (No.\! 2025A1515012248) and FY2025 MUSUBIME of Kyoto University.


\section*{Declarations}

\paragraph{Conflict of interest}
The authors declare that there are no conflicts of interest.




\begin{thebibliography}{99}

\bibitem{AG92}
Adams, E.E., Gelhar, L.W.: Field study of dispersion in a heterogeneous aquifer 2. Spatial moments analysis. Water Resources Res. {\bf28}(12), 3293--3307 (1992)

\bibitem{A75}
Adams, R.A.: Sobolev Spaces. Academic Press, New York (1975)

\bibitem{BL26}
BenSalah, M., Liu, Y.: Inverse $t$-source problem and a strict positivity property for coupled subdiffusion systems. Fract. Calc. Appl. Anal. {\bf29}(4), 44 pp. (2026).

\bibitem{EK04}
Eidelman, S.D., Kochubei, A.N.: Cauchy problem for fractional diffusion equations. J. Differential Equations {\bf199}(2), 211--255 (2004)

\bibitem{F99}
Folland, G.B.: Real Analysis: Modern Techniques and Their Applications (Second Edition). John Wiley \& Sons, New York (1999)

\bibitem{GCR92}
Giona, M., Cerbelli, S., Roman, H.E.: Fractional diffusion equation and relaxation in complex viscoelastic materials. Phys. A {\bf191}(1), 449--453 (1992)

\bibitem{GKMR14}
Gorenflo, R., Kilbas, A.A., Mainardi, F., Rogosin, S.V.: Mittag-Leffler Functions, Related Topics and Applications. Springer, Berlin (2014)

\bibitem{GLY15}
Gorenflo, R., Luchko, Y., Yamamoto, M.: Time-fractional diffusion equation in the fractional Sobolev spaces. Fract. Calc. Appl. Anal. {\bf18}(3), 799--820 (2015).

\bibitem{J21}
Jin, B.: Fractional Differential Equations: An Approach via Fractional Derivatives. Springer, Cham (2021)

\bibitem{JLZ13}
Jin, B., Lazarov, R., Zhou, Z.: Error estimates for a semidiscrete finite element method for fractional order parabolic equations. SIAM J. Numer. Anal. {\bf51}(1), 445--466 (2013)

\bibitem{JZ23}
Jin, B., Zhou, Z.: Numerical Treatment and Analysis of Time-Fractional Evolution Equations. Springer, Cham (2023)  

\bibitem{KR23}
Kaltenbacher, B., Rundell, W.: Inverse Problems for Fractional Partial Differential Equations. AMS, Providence, RI (2023)  

\bibitem{KRY20}
Kubica, A., Ryszewska, K., Yamamoto, M.: Time-Fractional Differential Equations: A Theoretical Introduction. Springer, Singapore (2020)

\bibitem{LHL23}
Li, Z., Huang, X., Liu, Y.: Initial-boundary value problems for coupled systems of time-fractional diffusion equations. Fract. Calc. Appl. Anal. {\bf26}(2), 533--566 (2023).

\bibitem{LLY15}
Li, Z., Liu, Y., Yamamoto, M.: Initial-boundary value problems for multi-term time-fractional diffusion equations with positive constant coefficients. Appl. Math. Comput. {\bf257}, 381--397 (2015)

\bibitem{LiLiuY19}
Li, Z., Liu, Y., Yamamoto, M.: Inverse problems of determining parameters of the fractional partial differential equations. in: Handbook of Fractional Calculus with Applications. Volume 2: Fractional Differential Equations, De Gruyter, Berlin, 431--442 (2019)

\bibitem{LiY19}
Li, Z., Yamamoto, M.: Inverse problems of determining coefficients of the fractional partial differential equations. in: Handbook of Fractional Calculus with Applications. Volume 2: Fractional Differential Equations, De Gruyter, Berlin, 443--464 (2019)  

\bibitem{LX07}
Lin, Y., Xu, C.: Finite difference/spectral approximations for the time-fractional diffusion equation. J. Comput. Phys. {\bf225}, 1533--1552 (2007)

\bibitem{LiuLiY19}
Liu, Y., Li, Z., Yamamoto, M.: Inverse problems of determining sources of the fractional partial differential equations. in: Handbook of Fractional Calculus with Applications. Volume 2: Fractional Differential Equations, De Gruyter, Berlin, 411--429 (2019)

\bibitem{LRY16}
Liu, Y., Rundell, W., Yamamoto, M.: Strong maximum principle for fractional diffusion equations and an application to an inverse source problem. Fract. Calc. Appl. Anal. {\bf19}(4), 888--906 (2016).

\bibitem{L09}
Luchko, Y.: Maximum principle for the generalized time-fractional diffusion equation. J. Math. Anal. Appl. {\bf351}(1), 218--223 (2009)

\bibitem{L10}
Luchko, Y.: Some uniqueness and existence results for the initial-boundary value problems for the generalized time-fractional diffusion equation. Comput. Math. Appl. {\bf59}(5), 1766--1772 (2010)

\bibitem{L19}
Luchko, Y.: The Wright function and its applications. in: Handbook of Fractional Calculus with Applications. Volume 1: Basic Theory, De Gruyter, Berlin, 241--268 (2019)

\bibitem{LuchkoY19}
Luchko, Y., Yamamoto, M.: Maximum principle for the time-fractional PDEs. in: Handbook of Fractional Calculus with Applications. Volume 2: Fractional Differential Equations, De Gruyter, Berlin, 299--326 (2019)

\bibitem{LY23}
Luchko, Y., Yamamoto, M.: Comparison principles for the time-fractional diffusion equations with the Robin boundary conditions. Part I: Linear equations. Fract. Calc. Appl. Anal. {\bf26}(4), 1504--1544 (2023).

\bibitem{LY25}
Luchko, Y., Yamamoto, M.: Comparison principles for the time-fractional diffusion equations with the Robin boundary conditions. Part II: Semilinear equations. Fract. Calc. Appl. Anal. {\bf28}(5), 2198--2240 (2025).

\bibitem{M10}
Mainardi, F.: Fractional Calculus and Waves in Linear Viscoelasticity: An Introduction to Mathematical Models. Imperial College Press, London (2010)

\bibitem{MK00}
Metzler, R., Klafter, J.: The random walk's guide to anomalous diffusion: a fractional dynamics approach. Phys. Rep. {\bf339}(1), 1--77 (2000)

\bibitem{SY11}
Sakamoto, K., Yamamoto, M.: Initial value/boundary value problems for fractional diffusion-wave equations and applications to some inverse problems. J. Math. Anal. Appl. {\bf382}(1), 426--447 (2011)

\bibitem{SSV12}
Schilling, R.L., Song, R., Vondracek, Z.: Bernstein Functions: Theory and Applications. De Gruyter, Berlin (2012)

\bibitem{Z13}
Zacher, R.: A weak Harnack inequality for fractional evolution equations with discontinuous coefficients. Ann. Sc. Norm. Super. Pisa Cl. Sci. {\bf12}(4), 903--940 (2013)

\end{thebibliography}
\end{document}